\documentclass[12pt]{amsart}

\usepackage{amsmath,amsthm}

\usepackage{amssymb}

\usepackage{pdfsync}

\newcommand{\R}{{\mathbb R}}       
\newcommand{\HH}{{\mathcal H}}

\newcommand{\diam}{{\rm diam}}
\newcommand{\dist}{{\rm dist}}

\newcommand{\rf}[1]{{(\ref{#1})}}
\newcommand{\supp}{{\rm supp\,}}

\newcommand{\ve}{{\varepsilon}}
\newcommand{\vv}{{\vspace{2mm}}}

\newcommand{\wt}[1]{{\widetilde{#1}}}

\newcommand{\rest}{{\lfloor}}

\newtheorem{theorem}{Theorem}[section]
\newtheorem*{theorema*}{Theorem A}
\newtheorem*{theoremb*}{Theorem B}
\newtheorem*{theoremc*}{Theorem C}
\newtheorem*{theoremd*}{Theorem D}
\newtheorem{lemma}[theorem]{Lemma}

\newtheorem{propo}[theorem]{Proposition}

\theoremstyle{definition}

\newcommand{\ci}[1]{_{{}_{\scriptstyle#1}}}

\theoremstyle{remark}

\numberwithin{equation}{section}

\begin{document}

\title[The Riesz transform, rectifiability, and removability]{The Riesz transform, rectifiability, and removability for Lipschitz harmonic functions}

\author{Fedor Nazarov}
\address{Fedor Nazarov, Department of Mathematical Sciences, Kent State University, Kent, Ohio, USA}

\author{Xavier Tolsa}
\address{Xavier Tolsa, ICREA/Universitat Aut\`onoma de Barcelona, Barcelona, Catalonia}

\author{Alexander Volberg}
\address{Alexander Volberg, Department of mathematics,  Michigan State University, East Lansing, Michigan, USA}

\thanks{F.N.\ was partially supported by the U.S.\ NSF grant DMS-0800243.
X.T.\ was partially supported by 
the ERC grant 320501 of the European Research
Council (FP7/2007-2013) 
and by the
grants
2009SGR-000420 (Catalonia) and MTM-2010-16232 (Spain). 
A.V.\ was partially supported by the U.S.\ NSF grant DMS-0758552}

\begin{abstract}
We show that, given a set $E\subset\R^{n+1}$ with finite $n$-Hausdorff measure $\HH^n$,
if the $n$-dimensional Riesz transform 
$$R_{\HH^n\rest E} f(x) = \int_{E} \frac{x-y}{|x-y|^{n+1}}\,f(y)d\HH^n(y)$$
is bounded in $L^2(\HH^n\rest E)$, then $E$ is $n$-rectifiable. From this result we deduce that
a compact set $E\subset\R^{n+1}$ with $\HH^n(E)<\infty$ is removable for
Lipschitz harmonic functions if and only if it is purely $n$-unrectifiable,  thus proving the analog of Vitushkin's conjecture in higher dimensions.
\end{abstract}

\maketitle

\section{Introduction}

In this paper we prove that, given a set $E\subset\R^{n+1}$ with finite $n$-Hausdorff measure $\HH^n$,
if the $n$-dimensional Riesz transform is bounded in $L^2$ with respect to $\HH^n\rest E$, then $E$ is $n$-rectifiable. Combined with results from \cite{Vo}, it implies that
the purely $n$-unrectifiable compact sets with finite $n$-Hausdorff measure are removable for
Lipschitz harmonic functions.

To state our results in more detail, we need to introduce some notation. Given a (complex) Borel measure
$\nu$ in $\R^{n+1}$ such that  
\begin{equation}
\label{infi}
\int_{\R^{n+1}} \frac{d |\nu|(x)}{(1+|x|)^n}<\infty\,,
\end{equation}
the $n$-dimensional Riesz transform of $\nu$ is defined by
$$R\nu(x) = \int \frac{x-y}{|x-y|^{n+1}}\,d\nu(y),$$
for every $x\in\R^{n+1}$ where the integral makes sense. Notice that the kernel inside the integral is vectorial.

Since the preceding integral may fail to be absolutely convergent
for many points $x\in\R^{n+1}$, it is convenient to consider an $\ve$-truncated version, for $\ve>0$:
$$R_\ve\nu(x) = \int_{|x-y|>\ve} \frac{x-y}{|x-y|^{n+1}}\,d\nu(y).$$
Given  a non-negative measure $\mu$ and $f\in L^1_{loc}(\mu)$ such that $\nu=f\mu$ satisfies \eqref{infi}, we set $R_\mu f(x) = R(f\,\mu)(x)$ and $R_{\mu,\ve} f(x) = R_\ve(f\,\mu)(x)$.
We say that $R_\mu$ is bounded in $L^2(\mu)$ if the truncated operators
$R_{\mu,\ve}$ are bounded in $L^2(\mu)$ uniformly on $\ve>0$.

A set $E\subset\R^{n+1}$ is called $n$-rectifiable if it is  contained
in a countable union of $C^1$-manifolds up to a set of zero $\HH^n$ measure.
On the other hand, $E$ is called purely unrectifiable it does not have any rectifiable subset with positive
$n$-Hausdorff measure.

Our first result in this paper is the following.

\vv
\begin{theorem}\label{teo1}
Let $E\subset\R^{n+1}$ be a set such that $\HH^n(E)<\infty$. 
If $R_{\HH^n\rest E}$ is bounded in $L^2(\HH^n\rest E)$, then $E$ is $n$-rectifiable.
\end{theorem}

\vv
Let us remark that the case $n=1$ of this theorem has already been 
known. Indeed,  the 
$1$-dimensional Riesz transform and the Cauchy transform coincide modulo a conjugation. So the
$L^2(\mu)$ boundedness of $R_\mu$ for any non-atomic measure $\mu$ 
implies that the curvature of $\mu$ is finite, i.e.,
$$
c^2(\mu)=\iiint \frac1{R(x,y,z)^2}\,d\mu(x)\,d\mu(y)\,d\mu(z)<\infty\,,
$$
where $R(x,y,z)$ stands for the radius of the circumference passing through $x,y,z$
(see \cite{Melnikov} and \cite{MV}).
It remains to refer to the result of David and L\'eger (see \cite{Leger}) who showed that
 if $\HH^1(E)<\infty$ and $c^2(\HH^1\rest E)<\infty$,
then $E$ is $1$-rectifiable.

In the higher dimensional setting, Theorem \ref{teo1} was an open problem (see \cite[p.\ 114]{Pajot},
for example). The main reason is that 
the curvature method is not available because for $n>1$, there is no relationship
between the $n$-dimensional Riesz transform and any notion as useful as curvature.
However, some partial results were known. For example, in \cite{Tolsa-jfa} it was shown that the
existence of the principal values $\lim_{\ve\to0}R_{\ve}(\HH^n\rest E)(x)$ for $\HH^n$-a.e.\ $x\in E$ implies the $n$-rectifiability of $E$ (see also \cite{Mattila-Preiss} for a previous result under somewhat stronger assumptions). On the other hand, from some new results in \cite{ENV} it follows that 
the Riesz transform $R_{\HH^n\rest E}$ cannot be bounded in $L^2(\HH^n\rest E)$ if $E$ is a  set with finite $\HH^n$ measure and  vanishing lower $n$-dimensional density $\theta^{n}_*(x,\HH^n\rest E)=\liminf_{r\to0}r^{-n}\HH^n(B(x,r)\cap E)$ for $\HH^n$-a.e\ $x\in E$.

Very recently, the proof of the so called David-Semmes
conjecture in the codimension $1$ case was completed. The conjecture follows from the results
of our new paper \cite{NToV} and  deep results by David and Semmes
in \cite{DS2}.
 The assertion is that if $\HH^n\rest E\subset \R^{n+1}$ is Ahlfors--David regular,
and $R_{\HH^n\rest E}$ is bounded in $L^2(\HH^n\rest E)$, then $E$ is uniformly $n$-rectifiable and
thus, in particular, $n$-rectifiable. 
Recall that a measure $\mu$ is called Ahlfors--David (AD) regular 
if
there exists some constant $c$ such that
$$
c^{-1}r^n\leq \mu(B(x,r))\leq c\,r^n\qquad \mbox{for all $x\in E$, $r>0$}\,.
$$

A nice covering theorem due to Pajot \cite{Pajot1} will allow us to reduce Theorem \ref{teo1} to the  combination of this result with the one from \cite{ENV}.

Let us remark that we do not know if Theorem \ref{teo1} can be extended to codimensions higher than $1$, that
is, to the $n$-dimensional Riesz transform and sets $E\subset \R^d$ with $0<\HH^n(E)<\infty$, $d>n+1$.
This is due to the fact that the corresponding analogs of the results from \cite{NToV} and \cite{ENV} are also open in this case.

Theorem \ref{teo1} has a corollary regarding the removability of singularities for Lipschitz harmonic functions.
Recall that a subset $E\subset\R^{n+1}$ is removable for Lipschitz harmonic functions if, for each open set $
\Omega\subset\R^{n+1}$, every Lipschitz function $f:\Omega\to\R$ that is harmonic in $\Omega\setminus E$ is 
harmonic in the whole $\Omega$. By combining Theorem \ref{teo1} with the results on the Lipschitz harmonic 
capacity from \cite{Vo}, one gets the proof of the following analog of Vitushkin's conjecture in higher
dimensions.

\vv
\begin{theorem}\label{teo2}
Let $E\subset\R^{n+1}$ be a compact set such that $\HH^n(E)<\infty$.
Then, $E$ is removable for Lipschitz harmonic functions in $\R^{n+1}$ if and only if $E$
is purely $n$-unrectifiable.
\end{theorem}

\vv
Again, this theorem was already known in the case $n=1$. 
It was proved by David and Mattila in \cite{David-Mattila}. The analogous result for the removable
singularities for bounded analytic functions is the celebrated solution of Vitushkin's conjecture
by David \cite{David-vitus}. It is also worth mentioning that in \cite{NTrV2},
a $Tb$ type theorem suitable to prove the analytic part of Vitushkin's conjecture was obtained
very shortly after David's proof. The arguments in \cite{NTrV2}
 also work in the case $n>1$ and are an essential tool in the work  \cite{Vo}
 about Lipschitz harmonic capacity.


\section{The main lemma}

\subsection{Statement of the Main Lemma}
We say that a Borel measure $\mu$ in $\R^d$ has growth  of degree $n$ if there exists some constant $c$ such
that 
$$
\mu(B(x,r))\leq c\,r^n\qquad \mbox{for all $x\in\R^d$, $r>0$.}
$$

We define  the upper and lower $n$-dimensional densities by
$$\theta^{n,*}(x,\mu)=\limsup_{r\to0}r^{-n}\mu(B(x,r))\text{ and } \theta^{n}_*(x,\mu)=\liminf_{r\to0}r^{-n}\mu(B(x,r)),$$
respectively.

If $\mu$ and $\sigma$ are Borel measures on $\R^d$, the notation $\mu\leq \sigma$ means that
$\mu(A)\leq\sigma(A)$ for all Borel sets $A\subset\R^d$.


\begin{lemma}[Main Lemma]\label{teocov}
Let $\mu$ be a compactly supported finite Borel measure in $\R^d$ with growth of degree $n$ such that 
$\theta^{n}_*(x,\mu)>0$ for $\mu$-a.e. $x\in\R^d$. Suppose that
$R_\mu$ is bounded in $L^2(\mu)$. Then there are finite Borel measures $\mu_k$, $k\geq 1$, such that 
\begin{itemize}
\item[(a)] $\mu \leq \sum_{k\geq 1}\mu_k$
\item[(b)] $\mu_k$ is AD-regular for each $k\geq1$ (with the AD-regularity constant depending on $k$), and
\item[(c)] for each $k\geq1$, $R_{\mu_k}$ is bounded in $L^2(\mu_k)$.
\end{itemize}
\end{lemma}

\vv

Before proving the Main Lemma, we will show  how it allows one to reduce Theorem \ref{teo1} 
 to the results in \cite{ENV} and \cite{NToV}.

As usual in harmonic analysis, the letter $c$ stands for some fixed constant (quite often an absolute constant),
which may change its value at different occurrences. On the other hand, constants with subscripts, such as
$c_1$, are assumed to keep their values in the whole paper.


\subsection{Proof of Theorem \ref{teo1} using the Main Lemma \ref{teocov}}

It is immediate to check that to prove the theorem, we may assume $E$ to be bounded.
So let $E\subset\R^{n+1}$ be a bounded set with $\HH^n(E)<\infty$. Set $\mu= \HH^n\rest E$, and suppose that 
$R_\mu$ is bounded in $L^2(\mu)$. 

Let $E_0$ be the subset of those $x\in E$ for which $\theta^n_*(x,\mu)=0$.
We set 
$$
\mu_0=\mu\rest E_0.
$$
Then,
$$
\theta^n_*(x,\mu_0)\le \theta^{n,*}(x,\mu)=0 \,\,\,\mbox{for $\mu_0$-a.e.\ $x\in\R^{n+1}$},
$$ 
and, moreover, $R_{\mu_0}$ is bounded in $L^2(\mu_0)$. Then, by the main theorem of \cite{ENV} (applied to the codimension $1$ case) we deduce that $\mu_0=0$.
That is, 
$$\theta^{n}_*(x,\mu)>0 \qquad \mbox{for $\mu$-a.e.\ $x\in\R^{n+1}$}.$$
So the measure $\mu$ satisfies the assumptions of Main Lemma \ref{teocov}, and thus we may consider
measures $\mu_k$ as in the statement of the Main Lemma.

 By the results  of  \cite{NToV}and \cite{DS2}, $\supp\mu_k$ is $n$-rectifiable. Therefore,
$$F=\bigcup_{k\geq1} \supp\mu_k$$
is also $n$-rectifiable. Since
$$\HH^n(E\setminus F) = \mu(\R^d\setminus F)\leq \sum_k \mu_k(\R^d\setminus F) =0,$$
we infer that $E$ is $n$-rectifiable too.


\section{Proof of the Main Lemma \ref{teocov}}


For the proof of the Main Lemma \ref{teocov} we will need the following proposition.

\begin{propo}\label{proposuma}
Let $\mu$ and $\sigma$ be Borel measures with growth of degree $n$ in $\R^d$ such that $R_\mu$ is bounded in $L^2(\mu)$ and $R_\sigma$ is bounded in $L^2(\sigma)$. Then, $R_{\mu+\sigma}$
is bounded in $L^2(\mu+\sigma)$.
\end{propo}

\begin{proof}
The boundedness of $R_\mu$ in $L^2(\mu)$ implies the boundedness of $R$ from the space of real measures
$M(\R^{d})$ into $L^{1,\infty}(\mu)$. In other words, the following inequality holds for
any $\nu\in M(\R^{d})$ uniformly on $\ve>0$:
$$\mu\bigl\{x\in\R^{d}:|R_{\ve}\nu(x)|>\lambda\bigr\}\leq c\,\frac{\|\nu\|}\lambda
\qquad\mbox{for all $\lambda>0$.}$$
For the proof, see Theorem 9.1 of \cite{NTrV1}.
Analogously, the same bound holds with $\mu$ replaced by $\sigma$. As a consequence, we infer that for all $\lambda>0$,
$$(\mu+\sigma)\bigl\{x\in\R^{d}:|R_{\ve}\nu(x)|>\lambda\bigr\}\leq c\,\frac{\|\nu\|}\lambda.$$
That is,  $R$ is bounded from $M(\R^d)$ into $L^{1,\infty}(\mu+\sigma)$. In particular, 
$R_{\mu+\sigma}$ is of weak type $(1,1)$ with respect to $\mu+\sigma$.
This implies that $R_{\mu+\sigma}$ is bounded in $L^2(\mu+\sigma)$. For the proof, based
on interpolation, see Theorem 10.1 of \cite{NTrV1} (an alternative argument based on a good lambda
inequality can be also found in Chapter 2 of the book \cite{Tolsa-book}).
\end{proof}

Let us remark that the preceding proposition and its proof  remain valid for more general Calder\'on-Zygmund operators.
However, we will need it only for the Riesz transforms.

\vv
In the proof of the Main Lemma \ref{teocov} it will  be convenient to work with an $\ve$-regularized
version $\wt R_{\mu,\ve}$ of the  Riesz transform $R_{\mu}$. We set
$$
\wt R_{\mu,\ve}f(x) = \int \frac{x-y}{\max(|x-y|,\ve)^{n+1}}\,f(y)\,d\mu(y).
$$
It is easy to check that
$$|\wt R_{\mu,\ve}f(x) -  R_{\mu,\ve}f(x)|\leq c\,M_\mu f(x)\qquad\mbox{for all $x\in\R^d$},$$
where $c$ is independent of $\ve$ and
$M_\mu$ is the centered maximal Hardy-Littlewood operator with respect to $\mu$:
$$M_\mu f(x)=\sup_{r>0} \frac1{\mu(B(x,r))}\int_{B(x,r)}|f|\,d\mu.$$
Since $M_\mu$ is bounded in $L^2(\mu)$, it turns out that
$R_\mu$ is bounded in $L^2(\mu)$ if and only if the operators $\wt R_{\mu,\ve}$ are bounded in $L^2(\mu)$
uniformly on $\ve>0$. The advantage of $\wt R_{\mu,\ve}$ over $R_{\mu,\ve}$ is that the kernel
$$K_\ve(x) = \frac{x}{\max(|x|,\ve)^{n+1}}$$
is continuous and
satisfies the smoothness condition
$$|\nabla K_\ve(x)|\leq \frac{c}{|x|^{n+1}},\quad |x|\neq\ve$$
(with $c$ independent of $\ve$),
which implies that $K_\ve(x-y)$
is a Calder\'on-Zygmund kernel (with constants independent of $\ve$), unlike the kernel of $R_{\mu,\ve}$.


\begin{proof}[\bf Proof of the Main Lemma \ref{teocov}]

We follow an idea of H. Pajot (see Theorem 10
of \cite{Pajot}), where some measures $\mu_k$ satisfying (a) and (b) are constructed.
For the reader's convenience, we will repeat  the arguments of the construction and of the proof
of (b), and subsequently we will show that
the statement (c) holds.


Consider the subset $F\subset\supp\mu$ of those $x\in\R^d$ for which $\theta^{n}_*(x,\mu)>0$, so that
$\mu(\R^d\setminus F)=0$. For  positive
integers $p,s$, we denote
\begin{align*}
F_p &=\left\{x\in F:\mbox{ for $0<r\leq  D$, $\mu(B(x,r))\geq \frac1p \,r^n$}\right\}, \\
F_{p,s} &=\left\{x\in F_p:\mbox{ for $0<r\leq  D$, $\mu(F_p\cap B(x,r))\geq \frac1{ps} \,r^n$}\right\}\,,
\end{align*}
where $D=\diam(\supp \mu)$.
From the definitions of $F$ and $F_p$, it is clear that  
$$
F=\bigcup_{p\geq 1} F_p.
$$
Also, $\theta_*^n(x,\mu) = \theta_*^n(x,\mu\rest F_p)$ for $\mu$-a.e.\ $x\in F_p$ by the Lebesgue 
differentiation theorem, and thus
$$
\mu\Bigl(F_p\setminus \bigcup_{s\geq1} F_{p,s}\Bigr)=0.
$$
So we have
$$
\mu\leq \sum_{p,s\geq1} \mu\rest F_{p,s}.
$$

The strategy of the construction consists in adding a measure $\sigma_{p,s}$ to each $\mu\rest F_{p,s}$ so that
the resulting measure is AD-regular, for each $p,s$.

It is easy to check that all the sets $F_p$  and $F_{p,s}$ are compact.  Fix $p,s$ and denote
$$
d(x) = \frac1{10}\,\dist(x,F_{p,s}).
$$
Notice that $d(y)>0$ if $x\not\in F_{p,s}$, as $F_{p,s}$ is closed.
Now we cover $F_p\setminus F_{p,s}$ by a family of balls of the form $B(x,d(x))$, with $x\in F_p\setminus
F_{p,s}$, using Besicovitch's covering theorem. So there exists a family of points $H_{p,s}\subset F_p\setminus F_{p,s}$, at most countable, such that 
$$
F_p\setminus F_{p,s}\subset\bigcup_{x\in H_{p,s}} B(x,d(x)),
$$
and
$$\sum_{x\in H_{p,s}} \chi\ci{B(x,d(x))}\leq C_d\,.$$
 Moreover, we can split $H_{p,s}=\bigcup_{i=1}^{N_d} H_{p,s}^i$ so that
for each $i$, the balls from $\{B(x,d(x))\}_{x\in H_{p,s}^i}$ are pairwise disjoint (see Theorem 2.7 in p.\ 30 of \cite{Mattila-llibre}) . Here   $C_d, N_d$ are some constants depending on $d$ only.

To define $\sigma_{p,s}$, for each $x\in H_{p,s}$ we consider an arbitrary $n$-plane $\Pi_x$ containing 
$x$ and  set $P_x=\Pi_x\cap B(x,\frac12 d(x))$. Then we define
$$
\sigma_{p,s} = \HH^n\rest \Pi_{p,s} +\sum_{x\in H_{p,s}} \HH^n\rest P_x
$$
where $\Pi_{p,s}$ is an arbitrary $n$-plane in $\R^d$ intersecting $F_{p,s}$.
We set
$$\mu_{p,s}= \sigma_{p,s} + \mu\rest F_{p,s}.$$
We also denote
$$\sigma_{p,s}^i = \sum_{x\in H_{p,s}^i} \HH^n\rest P_x,$$
so that $\sigma_{p,s} = \HH^n\rest \Pi_{p,s} +\sum_{i=1}^{N_d}\sigma_{p,s}^i$.
We will show now that $\mu_{p,s}$ is AD-regular.

\vspace{3mm}
\noindent
{\bf Upper AD-regularity of  $\mu_{p,s}$. }
We have to show that $\mu_{p,s}$ has growth of degree $n$. Since $\mu\rest F_{p,s}$ and $\HH^n\rest\Pi$ have growth of
degree $n$, it is enough to show that so does $\sigma_{p,s}^i$, for each 
$i=1,\ldots,N_d$.

We set  $r(x)=\frac12\,d(x)$ for $x\in H_{p,s}^i$. So $\sigma_{p,s}^i$ is a measure supported on the union of the closed balls 
\begin{equation}\label{eqbj}
B_x:=B(x,r(x))= B(x,\frac12\,d(x)),\quad \mbox{ $x\in H_{p,s}^i$,}
\end{equation}
coinciding with $\HH^n\rest P_x$   inside $B_x$. Notice also that
the balls $2B_x$, $x\in  H_{p,s}^i$ are pairwise disjoint.

Let $\Delta$ be some fixed closed ball of radius $r(\Delta)$. Let $H_a$ be the subset of points $x\in H_{p,s}^i$ such that
$B_x\cap \Delta\neq\varnothing$ and $2r(\Delta)<r(x)$, and  let $H_b\subset H_{p,s}^i$ be the subset of points such that
$B_x\cap \Delta\neq\varnothing$ and $2r(\Delta)\geq r(x)$. We have
$$
\sigma_{p,s}^i(\Delta) = \sum_{x\in H_a}\HH^n(P_{x}\cap\Delta) + \sum_{x\in H_b}\HH^n(P_{x}\cap \Delta),$$

It is immediate that for $x\in H_a$ we have $\Delta\subset 2B_x$. Thus, since the balls $2B_x$ are
pairwise disjoint, $H_a$ contains at most one point $z$, and so,
$$\sum_{x\in H_a}\HH^n(P_{x}\cap\Delta)=\HH^n(P_z\cap\Delta) \leq c\, r(\Delta)^n.$$
On the other hand, for $x\in H_b$, we have $B_x\subset 5\Delta$. Recall also
that, since $B_j$ is centered at some point from $F_p$, 
\begin{equation}\label{eqda3}
\mu(B_x)\geq \frac1p\,r(x)^n
\end{equation}
(recall that $r(x) =\frac12 d(x)\le D$ because $d(x)$ is the distance between two points in $\supp\mu$).
As a consequence,
$$\sum_{x\in H_b}\HH^n(P_{x}\cap \Delta)\leq c\sum_{x\in H_b}r(x)^n \leq cp\sum_{x\in H_b}\mu(B_x)
\leq cp\,\mu(5\Delta)\leq cp\,r(\Delta)^n,$$
by the growth of degree $n$ of $\mu$.

\vspace{3mm}
\noindent
{\bf Lower AD-regularity of  $\mu_{p,s}$. }
Consider an arbitrary closed ball $\Delta$ centered at $\supp\mu_{p,s}$. Suppose first that it is 
centered at some point $z\in F_{p,s}$.

 First we claim that if $2B_x$, $x\in H_{p,s}$,
intersects $\frac12 \Delta$, then $2B_x\subset \Delta$. Indeed, recalling that the function $d(\cdot)$
is $\frac1{10}$-Lipschitz, we get
$$10\,d(x) = \dist(x,F_{p,s}) \leq |x-z| \leq \frac12\,r(\Delta) + 2r(x) =\frac12\,r(\Delta) + d(x).$$
Thus, 
$$r(2B_x) = 2r(x) =d(x)\leq \frac{1}{18}\,r(\Delta),$$
which clearly implies that $2B_x\subset \Delta$. 

From the claim above and the definition of $\mu_{p,s}$ we deduce that
$$
\mu_{p,s}(\Delta)\geq \mu(F_{p,s}\cap \Delta) + c^{-1}\sum_{x\in H_{p,s}:2B_x\cap \frac12\Delta\neq\varnothing}
r(x)^n.$$
Since $\mu$ has growth 
of degree $n$, and since 
$$
\bigcup_{x\in H_{p,s}:2B_x\cap \frac12\Delta\neq\varnothing}2B_x \supset F_p\cap\frac12\Delta\setminus F_{p,s}\,,
$$ 
we derive
\begin{align*}
\mu_{p,s}(\Delta)& \geq \mu(F_{p,s}\cap \Delta) + c^{-1}\sum_{x\in H_{p,s}:2B_x\cap \frac12\Delta\neq\varnothing}
\mu(2B_x)\\
& \geq \mu\Bigl(F_{p,s}\cap \frac12 \Delta\Bigr) + c^{-1}\mu\Bigl(F_p\cap\frac12\Delta\setminus F_{p,s}\Bigr)
\geq c^{-1} \mu\Bigl(F_p\cap \frac12 \Delta\Bigr).
\end{align*}
Since $\Delta$ is centered at some point from $F_{p,s}$, 
$$\mu\Bigl(F_p\cap \frac12 \Delta\Bigr)\geq c^{-1}\,\frac1{ps}\,r(\Delta)^n,$$
provided that $r(\Delta)\leq 2D$. On the other hand, if $r(\Delta) \ge 2D$  then $\dist ( z, \Pi_{p,s})\le D$, and, thereby, 
$$
\mu_{p,s} (\Delta) \geq \HH^n (B(z', \frac12r(\Delta)) \geq c\, r(\Delta)^n\,,
$$
where $z'$ is the nearest to $z$ point of $\Pi_{p,s}$.

\vv
If $\Delta$ is centered on $\Pi_{p,s}$, the lower bound is trivial.

\vv
Suppose now that $\Delta$ is centered at some point  $z\in P_{x}$, for some $x\in H_{p,s}$.
If $r(\Delta)\leq 40 r(B_x)$, from the lower AD-regularity of $\HH^n\rest P_x$ we infer that
$\mu(\Delta)\geq c^{-1}\,r(\Delta)^n$. Assume now that $r(\Delta)> 40 r(x) =20d(x)$. In this case,
by the definition of $d(x)$, there exists some $y\in F_{p,s}$
satisfying
$$|z-y|\leq |z-x| + |x-y|\leq r(x)+ 10d(x) = \frac{21}{2}\,d(x) \leq \frac{21}{40}\,r(\Delta).$$
Thus, $\Delta$ contains the ball $B(y,\frac1{10}\,r(\Delta))$. Then, since 
$$\mu\Bigl(B\Bigl(y,\frac1{10}\,r(\Delta)\Bigr)\Bigr) \geq c^{-1}\,\frac1{ps}\,r(\Delta)^n,$$
we are done.

\vspace{3mm}
\noindent
{\bf Boundedness of $R_{\mu_{p,s}}$ in $L^2(\mu_{p,s})$. } 
Taking into account that $R_{\mu\rest F_{p,s}}$ is bounded in $L^2(\mu\rest F_{p,s})$, and  that $R_{\HH^n\rest\Pi_{p,s}}$ is bounded in $L^2({\HH^n\rest\Pi_{p,s}})$, it is enough to show that $R_{\sigma_{p,s}^i}$ is bounded in $L^2(\sigma_{p,s}^i)$ for each $i=1,\ldots,N_d$. Then the repeated application
of Proposition \ref{proposuma} yields the result.

To simplify notation, for fixed $p,s,i$, we denote $\sigma=\sigma_{p,s}^i$, $H=H_{p,s}^i$. 

Now we define
$$\nu = \sum_{x\in H} c_x\,\mu\rest B_x,$$
with $c_x=\HH^n(P_{x})/\mu(B_x)$. Observe that the constants $c_x$, $x\in H$, are uniformly bounded by some constant depending on $p$, because of \rf{eqda3},
and thus $R_\nu$ is bounded in $L^2(\nu)$. Further,
$\nu(B_x) = \sigma(B_x)$ for each $x\in H$. Recall also that, by construction both $\sigma$ and $\nu$
are supported on the union of the balls $B_x, x\in H,$ and the double balls $2B_x$ are pairwise disjoint.

It is clear that, in a sense, $\nu$ can be considered as an approximation of $\sigma$ (and
 conversely). To prove the boundedness of $R_\sigma$ in $L^2(\sigma)$, 
 we will prove that $\wt R_{\sigma,\ve}$ is bounded in $L^2(\sigma)$ uniformly on $\ve>0$ by comparing 
 it to $\wt R_{\nu,\ve}$.
First we need to introduce some local and non local operators: given $z\in \bigcup_{x\in H} B_x$, we denote by
$B(z)$ the ball $B_x, x\in H,$ that contains $z$. Then we write, for $z\in\bigcup_{x, x\in H} B_x$,
$$R^{loc}_{\nu,\ve} f(z) = \wt R_{\nu,\ve}(f\chi\ci{B(z)})(z),
 \qquad R^{nl}_{\nu,\ve} f(z) = \wt R_{\nu,\ve}(f\chi\ci{\R^d \setminus B(z)})(z).$$
We define analogously $R^{loc}_{\sigma,\ve} f$ and $R^{nl}_{\sigma,\ve} f$.
It is straightforward to check that $R^{loc}_{\nu,\ve}$ is bounded in $L^2(\nu)$, and
that $R^{loc}_{\sigma,\ve}$ is bounded in $L^2(\sigma)$, both uniformly
on $\ve$ (in other words, 
$R^{loc}_{\nu}$ is bounded in $L^2(\nu)$ and
 $R^{loc}_{\sigma}$ is bounded in $L^2(\sigma)$). 
Indeed,
$$\|R^{loc}_{\sigma,\ve}f\|_{L^2(\sigma)}^2 = \sum_{x\in H}
\|\chi\ci{B_x}\wt R_{\sigma,\ve}(f\chi\ci{B_x})\|_{L^2(\sigma)}^2 \leq c\sum_{x\in H}
\|f\chi\ci{B_x}\|_{L^2(\sigma)}^2 = c\|f\|_{L^2(\sigma)}^2,$$
by the boundedness of the $n$-Riesz transforms on $n$-planes. Using the boundedness
of $R_\nu$ in $L^2(\nu)$, one derives the $L^2(\nu)$ boundedness of $R^{loc}_{\nu,\ve}$ analogously.

\vspace{2mm}
\noindent
{\bf Boundedness of $R^{nl}_{\sigma}$ in $L^2(\sigma)$. }
We must show that 
$R^{nl}_{\sigma}$ is bounded in $L^2(\sigma)$. Observe first that, since 
$R^{nl}_{\nu,\ve} = \wt R_{\nu,\ve} - R^{loc}_{\nu,\ve}$, and both
$\wt R_{\nu,\ve}$ and $R^{loc}_{\nu,\ve}$ are bounded in $L^2(\nu)$, it turns out that 
$R^{nl}_{\nu,\ve}$ is bounded in $L^2(\nu)$ (all uniformly on $\ve>0$).

We will prove below that, for all $f\in L^2(\nu)$ and $g\in L^2(\sigma)$ satisfying
\begin{equation}\label{eqcond33}
\int_{B_x} f\,d\nu = \int_{B_x} g\,d\sigma\qquad \mbox{for all $x\in H$,}
\end{equation}
we have
\begin{equation}\label{eqkey1}
I(f,g) := \int |R^{nl}_{\nu,\ve} f - R^{nl}_{\sigma,\ve} g|^2\,d(\nu + \sigma) \leq
c\,(\|f\|_{L^2(\nu)}^2 + \|g\|_{L^2(\sigma)}^2),
\end{equation} 
uniformly on $\ve$.
Let us see how the boundedness of $R^{nl}_{\sigma}$ in $L^2(\sigma)$ follows from this
 estimate.
As a preliminary step, we show that $R_\sigma^{nl}:L^2(\sigma)\to L^2(\nu)$ is bounded. To this
end, given $g\in L^2(\sigma)$, we consider a function $f\in L^2(\nu)$ satisfying \rf{eqcond33}
that is constant on each ball $B_j$. It is straightforward to check that
$$
\|f\|_{L^2(\nu)}\leq \|g\|_{L^2(\sigma)}.
$$
Then from the $L^2(\nu)$ boundedness of $R_\nu^{nl}$ and \rf{eqkey1}, we obtain
$$
\|R_{\sigma,\ve}^{nl}g\|_{L^2(\nu)}\leq \|R_{\nu,\ve}^{nl}f\|_{L^2(\nu)} + I(f,g)^{1/2}\leq 
c \|f\|_{L^2(\nu)} + c\,\|g\|_{L^2(\sigma)}\leq  c\,\|g\|_{L^2(\sigma)},
$$
which proves that $R_\sigma^{nl}:L^2(\sigma)\to L^2(\nu)$ is bounded. 

Notice that $R^{nl}_\ve$ is antisymmetric. Indeed, its kernel is
$$
\left[ 1-\sum_{x\in H} \chi\ci{B_x}(z)\chi\ci{B_x}(y)\right]\,\frac{z-y}{\max(|z-y|,\ve)^{n+1}}\,.
$$
Then, by duality, we deduce that
$R_\nu^{nl}:L^2(\nu)\to L^2(\sigma)$ is bounded.
To prove now the $L^2(\sigma)$ boundedness of $R_\sigma^{nl}$, we consider an arbitrary function
$g\in L^2(\sigma)$, and we construct $f\in L^2(\nu)$ satisfying \rf{eqcond33} which is constant
in each ball $B_x$.  Again, we have $\|f\|_{L^2(\nu)}\leq \|g\|_{L^2(\sigma)}.$

Using the boundedness of $R_\nu^{nl}:L^2(\nu)\to L^2(\sigma)$ together with \rf{eqkey1}, we obtain
$$\|R_{\sigma,\ve}^{nl}g\|_{L^2(\sigma)}\leq \|R_{\nu,\ve}^{nl}f\|_{L^2(\sigma)} + I(f,g)^{1/2}\leq 
c \|f\|_{L^2(\nu)} + c\,\|g\|_{L^2(\sigma)}\leq  c\,\|g\|_{L^2(\sigma)},$$
as wished. 

It remains to prove that \rf{eqkey1} holds for $f\in L^2(\nu)$ and $g\in L^2(\sigma)$ satisfying
\rf{eqcond33}.
For $z\in\bigcup_{x\in H} B_x$, we have
$$|R^{nl}_{\nu,\ve} f(z) - R^{nl}_{\sigma,\ve} g(z)| \leq \sum_{x\in H:z\not\in B_x} \left|\int_{B_x} K_\ve(z-y)
(f(y)\,d\nu(y)-g(y)\,d\sigma(y))\right|,$$
where $K_\ve(z)$ is the kernel of the $\ve$-regularized  $n$-Riesz transform.
By standard estimates, using \rf{eqcond33}, the fact that the balls $2B_x$, $x\in H$, are pairwise disjoint,
and the smoothness of $K_\ve$, it follows that
\begin{align*}
\biggl|\int_{B_x} K_\ve&(z-y)
(f(y)\,d\nu(y)-g(y)\,d\sigma(y))\biggr| \\ &= \left|\int_{B_x} (K_\ve(z-y)- K_\ve(z-x))
(f(y)\,d\nu(y)-g(y)\,d\sigma(y))\right|\\
&\leq c\int_{B_x}\frac{|x-y|}{|x-y|^{n+1}}(|f(y)|\,d\nu(y)+|g(y)|\,d\sigma(y))\\
&\approx \frac{r(x)}{\dist(B(z),B_x)^{n+1}}\,\int_{B_x}(|f|\,d\nu+|g|\,d\sigma).
\end{align*}
Recall that $B(z)$ stands for the ball $B_x, x\in H,$ that contains $z$.

We consider the operators
$$
T_\nu(f)(z) = \sum_{x\in H: z\notin B_x}\frac{r(x)}{\dist(B(z),B_x)^{n+1}}\,\int_{B_x}
f\,d\nu\,,
$$
and $T_\sigma$, which is defined in the same way with $\nu$ replaced by $\sigma$.
Observe that
\begin{align*}
I(f,g)& \leq c\|T_\nu(|f|) + T_\sigma(|g|)\|_{L^2(\nu+\sigma)}^2 \\
&\leq 
2c\|T_\nu(|f|)\|_{L^2(\nu+\sigma)}^2 + 2c\|T_\sigma(|g|)\|_{L^2(\nu+\sigma)}^2 \\
&= 
4c\|T_\nu(|f|)\|_{L^2(\nu)}^2 + 4c\|T_\sigma(|g|)\|_{L^2(\sigma)}^2,
\end{align*}
where, for the last equality, we took into account that both $T_\nu(|f|)$ and $T_\sigma(|g|)$ are constant on each ball
$B_x$ and that $\nu(B_x)=\sigma(B_x)$ for all $x\in H$.

To finish the proof of \rf{eqkey1} it is enough to show that $T_\nu$ is bounded in $L^2(\nu)$ and 
$T_\sigma$ in $L^2(\sigma)$. We only deal with $T_\nu$, since the arguments for $T_\sigma$ are
analogous. We argue by duality again. So we consider non-negative functions $f,h\in L^2(\nu)$. We have
\begin{align*}
\int T_\nu(f)\,h\,d\nu &\approx \int \left(\sum_{x\in H:z\not\in B_x}\frac{r(x)}{\dist(z,B_x)^{n+1}}\,\int_{B_x}
f\,d\nu\right)\,h(z)\,d\nu(z)\\
&=
\sum_{x\in H}r(x) \int_{B_x}f\,d\nu  \int_{\R^d\setminus B_x}\frac1{\dist(z,B_x)^{n+1}}\,
\,h(z)\,d\nu(z).
\end{align*}
From the growth of degree $n$ of $\nu$ and the fact that the balls $2B_x$ are disjoint, it follows easily that
$$
\int_{\R^d \setminus B_x}\frac1{\dist(z,B_x)^{n+1}}\,
\,h(z)\,d\nu(z)\leq \frac{c}{r(x)}\,M_\nu h(y),
$$
for all $y\in B_x$, where $M_\nu$ stands for the (centered) maximal Hardy-Littlewood operator (with respect to $\nu$).
 Then we deduce that
$$\int T_\nu(f)\,h\,d\nu\lesssim \sum_{x\in H} \int_{B_x}f(y)\,M_\nu h(y)\,d\nu(y)\lesssim
\|f\|_{L^2(\nu)} \|h\|_{L^2(\nu)},$$
by the $L^2(\nu)$ boundedness of $M_\nu$. Thus $T_\nu$ is bounded in $L^2(\nu)$.
\end{proof}


\section{Proof of Theorem \ref{teo2}}

It is already known that if 
$E\subset\R^{n+1}$ with $\HH^n(E)
<\infty$  is removable for Lipschitz harmonic functions, then it must be purely 
$n$-unrectifiable (see \cite{MP}). So we only have to show that if $E$ is non removable, then
it is not purely $n$-unrectifiable.
The proof follows from Theorem \ref{teo1} and results from
\cite{Vo}, by standard arguments. However, for the reader's convenience we show all the details.

Let $E\subset\R^{n+1}$ be a compact set such that $\HH^n(E)<\infty$. Suppose
that it is not removable for Lipschitz harmonic functions. Then, by  Theorem 2.2 from
\cite{Vo}, there exists some measure $\mu$ supported on $E$ with growth of  degree $n$ such that
$R_\mu$ is bounded in $L^2(\mu)$.

The growth condition on $\mu$ implies that $\mu\ll\HH^n\rest E$. Indeed, 
let $A\subset\R^{n+1}$ be a Borel set.  Consider $\ve>0$ and any covering $A\subset \bigcup_i A_i$ with $\diam A_i\leq \ve$. For each $i\geq1$, let $B_i$ be a closed ball of radius
equal to $\diam(A_i)$ centered at some
point from $A_i$. Since $A\subset\bigcup_i B_i$,  we have
$$\mu(A)\leq \sum_i \mu(B_i)\leq c_0\sum_i \diam(A_i)^n.$$
Taking the infimum over all the coverings $\bigcup_i A_i$ as above, we obtain
$$\mu(A)\leq c_0\,\HH^n(A) \quad\mbox{ for any Borel set $A\subset\R^d$.}$$
As a consequence,
\begin{equation}\label{eqrad1}
\mu(A)=\mu(A\cap E)\leq c_0\,\HH^n(A\cap E) \quad\mbox{ for any Borel set $A\subset\R^d$.}
\end{equation}
That is, $\mu\ll\HH^n\rest E$.

From the Radon-Nykodim theorem we infer that there exists some function $g\in L^1(\HH^n\rest E)$
such that $\mu=g\,\HH^n\rest E$. In fact, from \rf{eqrad1} and the Lebesgue differentiation theorem
we deduce that $\|g\|_\infty\leq c_0$. 
Take $\eta>0$ small enough so that 
$$F=\{x\in E: g(x)>\eta\}$$
satisfies $\mu(F)>0$ (and thus $\HH^n(F)>0$ too).
Since $\mu\rest F= g\,\HH^n\rest F$, we have
$$\HH^n\rest F = g^{-1}\mu\rest F.$$
As $g^{-1}$ is bounded by $\eta^{-1}$ in $F$, we deduce that $R_{\HH^n\rest F}$ is bounded
in $L^2(\HH^n\rest F)$. As a consequence, $F$ is $n$-rectifiable, by Theorem \ref{teo1}.
That is, $E$ is not purely $n$-unrectifiable.


\end{document}